\documentclass[11pt,a4paper]{amsart}
\usepackage[margin=25mm]{geometry}
\usepackage[numbers]{natbib}
\usepackage{hyperref}

\usepackage{amssymb,amsmath}
\usepackage{amsthm}
\usepackage{bbm, bm}
\usepackage{stmaryrd}
\usepackage[usenames,dvipsnames]{xcolor}
\usepackage{color}
\usepackage{graphicx}
\usepackage{caption}
\usepackage{float}
\usepackage{wrapfig}

\usepackage{lineno}

\usepackage{abstract}

\input xy
\xyoption{all} 

\numberwithin{equation}{section}

\newtheorem{theorem}{Theorem}[section]
\newtheorem{corollary}[theorem]{Corollary}

\newtheorem{proposition}[theorem]{Proposition}

\theoremstyle{definition}
\newtheorem{definition}[theorem]{Definition}
\newtheorem{example}[theorem]{Example}

\newenvironment{warning}[1][Warning.]{\begin{trivlist}
\item[\hskip \labelsep {\bfseries #1}]}{\end{trivlist}}

\newenvironment{remark}[1][Remark.]{\begin{trivlist}
\item[\hskip \labelsep {\bfseries #1}]  }{ \end{trivlist}}

\newcommand{\Id}{\mathbbmss{1}}

\newcommand{\rmd}{\textnormal{d}}
\newcommand{\rme}{\textnormal{e}}

\newcommand*{\longhookrightarrow}{\ensuremath{\lhook\joinrel\relbar\joinrel\rightarrow}}

\DeclareMathOperator{\Vect}{Vect}

\DeclareMathOperator{\Diff}{Diff}

\DeclareMathOperator{\Hom}{Hom}

\newcommand{\catname}[1]{\textnormal{\texttt{#1}}}

\font\black=cmbx10 \font\sblack=cmbx7 \font\ssblack=cmbx5 \font\blackital=cmmib10  \skewchar\blackital='177
\font\sblackital=cmmib7 \skewchar\sblackital='177 \font\ssblackital=cmmib5 \skewchar\ssblackital='177
\font\sanss=cmss10 \font\ssanss=cmss8 
\font\sssanss=cmss8 scaled 600 \font\blackboard=msbm10 \font\sblackboard=msbm7 \font\ssblackboard=msbm5
\font\caligr=eusm10 \font\scaligr=eusm7 \font\sscaligr=eusm5  \font\fraktur=eufm10
\font\sfraktur=eufm7 \font\ssfraktur=eufm5 
\font\bsymb=cmsy10 scaled\magstep2
\def\all#1{\setbox0=\hbox{\lower1.5pt\hbox{\bsymb
       \char"38}}\setbox1=\hbox{$_{#1}$} \box0\lower2pt\box1\;}
\def\exi#1{\setbox0=\hbox{\lower1.5pt\hbox{\bsymb \char"39}}
       \setbox1=\hbox{$_{#1}$} \box0\lower2pt\box1\;}

\def\tx#1{{\fam0\relax#1}}

\newfam\bifam
\textfont\bifam=\blackital \scriptfont\bifam=\sblackital \scriptscriptfont\bifam=\ssblackital

\newfam\blfam
\textfont\blfam=\black \scriptfont\blfam=\sblack \scriptscriptfont\blfam=\ssblack

\newfam\bbfam
\textfont\bbfam=\blackboard \scriptfont\bbfam=\sblackboard \scriptscriptfont\bbfam=\ssblackboard

\newfam\ssfam
\textfont\ssfam=\sanss \scriptfont\ssfam=\ssanss \scriptscriptfont\ssfam=\sssanss
\def\sss#1{{\fam\ssfam\relax#1}}

\newfam\clfam
\textfont\clfam=\caligr \scriptfont\clfam=\scaligr \scriptscriptfont\clfam=\sscaligr

\newfam\frfam
\textfont\frfam=\fraktur \scriptfont\frfam=\sfraktur \scriptscriptfont\frfam=\ssfraktur

\def\hpb#1{\setbox0=\hbox{${#1}$}
    \copy0 \kern-\wd0 \kern.2pt \box0}
\def\vpb#1{\setbox0=\hbox{${#1}$}
    \copy0 \kern-\wd0 \raise.08pt \box0}

\def\pmb#1{\setbox0\hbox{${#1}$} \copy0 \kern-\wd0 \kern.2pt \box0}
\def\pmbb#1{\setbox0\hbox{${#1}$} \copy0 \kern-\wd0
      \kern.2pt \copy0 \kern-\wd0 \kern.2pt \box0}
\def\pmbbb#1{\setbox0\hbox{${#1}$} \copy0 \kern-\wd0
      \kern.2pt \copy0 \kern-\wd0 \kern.2pt
    \copy0 \kern-\wd0 \kern.2pt \box0}
\def\pmxb#1{\setbox0\hbox{${#1}$} \copy0 \kern-\wd0
      \kern.2pt \copy0 \kern-\wd0 \kern.2pt
      \copy0 \kern-\wd0 \kern.2pt \copy0 \kern-\wd0 \kern.2pt \box0}
\def\pmxbb#1{\setbox0\hbox{${#1}$} \copy0 \kern-\wd0 \kern.2pt
      \copy0 \kern-\wd0 \kern.2pt
      \copy0 \kern-\wd0 \kern.2pt \copy0 \kern-\wd0 \kern.2pt
      \copy0 \kern-\wd0 \kern.2pt \box0}

\mathchardef\za="710B  
\mathchardef\zb="710C  
\mathchardef\zg="710D  
\mathchardef\zd="710E  
\mathchardef\zve="710F 
\mathchardef\zz="7110  
\mathchardef\zh="7111  
\mathchardef\zvy="7112 
\mathchardef\zi="7113  
\mathchardef\zk="7114  
\mathchardef\zl="7115  
\mathchardef\zm="7116  
\mathchardef\zn="7117  
\mathchardef\zx="7118  
\mathchardef\zp="7119  
\mathchardef\zr="711A  
\mathchardef\zs="711B  
\mathchardef\zt="711C  
\mathchardef\zu="711D  
\mathchardef\zvf="711E 
\mathchardef\zq="711F  
\mathchardef\zc="7120  
\mathchardef\zw="7121  
\mathchardef\ze="7122  
\mathchardef\zy="7123  
\mathchardef\zf="7124  
\mathchardef\zvr="7125 
\mathchardef\zvs="7126 
\mathchardef\zf="7127  
\mathchardef\zG="7000  
\mathchardef\zD="7001  
\mathchardef\zY="7002  
\mathchardef\zL="7003  
\mathchardef\zX="7004  
\mathchardef\zP="7005  
\mathchardef\zS="7006  
\mathchardef\zU="7007  
\mathchardef\zF="7008  
\mathchardef\zW="700A  
\mathchardef\zC="7009  

\newcommand{\be}{\begin{equation}}
\newcommand{\ee}{\end{equation}}

\newcommand{\bea}{\begin{eqnarray}}
\newcommand{\eea}{\end{eqnarray}}
\def\*{{\textstyle *}}
\newcommand{\R}{{\mathbb R}}

\newcommand{\C}{{\mathbb C}}

\newcommand{\s}{{\textstyle *}}

\def\Hom{\sss{Hom}}

\def\Vect{\sss{Vect}}

\def\sT{{\sss T}}

\def\sS{{\sss S}}

\def\xi{\tx{i}}

\def\s*{{\scriptstyle *}}

\newcommand{\beas}{\begin{eqnarray*}}
\newcommand{\eeas}{\end{eqnarray*}}

\title{On Lie semiheaps and ternary  principal bundles} 
\author{Andrew James Bruce } 
   \address{Department of Mathematics,
The Computational Foundry,
Swansea University Bay Campus,
Fabian Way,
Swansea, SA1 8EN}  
   \email{a.j.bruce@swansea.ac.uk, ~andrewjamesbruce@googlemail.com}

   \date{\today}
 
\begin{document}
 \maketitle
\vspace{-20pt}
\begin{abstract}{\noindent We introduce the notion of a Lie semiheap as a smooth manifold equipped with a para-associative ternary product.  For a particular class of Lie semiheaps we establish the existence of left-invariant vector fields. Furthermore, we show how such manifolds are related to Lie groups and establish the analogue of principal bundles in this ternary setting. In particular, we generalise the well-known `heapification' functor to the ambience of Lie groups and principal bundles.}\\

\noindent {\Small \textbf{Keywords:} Heaps;~Semiheaps;~Principal Bundles;~Group Actions;~Generalised Associativity}
\end{abstract}

\section{Introduction and Motivation}
\subsection{Introduction and Background} 
 Pr\"{u}fer \cite{Prufer:1924} and  Baer \cite{Baer:1929} introduced the notion of a \emph{heap} (also known as a \emph{torsor} or \emph{groud} or \emph{herd}) as a set with a ternary operation satisfying some natural axioms including a generalisation of associativity. A heap should be thought of as a group in which the identity element is absent. Given a group,  it can be turned into a heap by defining the ternary operation as $(x,y,z) \mapsto xy^{-1}z$. In fact, up to isomorphism, every heap  arises from a group in this way.  Conversely, by selecting an element in a heap, one can reduce the ternary operation to a  group operation, such that the chosen element is the identity element. However, we do not quite have an equivalence of categories here as passing from a heap to a group is not natural (there is a lot of choice here). That said, there is an isomorphism of categories between pointed heaps and groups. We will show that this isomorphism carries over to the smooth setting.\par 
A \emph{semiheap} (sometimes also referred to as a \emph{semitorsor}) is a  set $S$, equipped with a ternary operation $[x,y,z] \in S$ that satisfies the \emph{para-associative law} 
\begin{equation*}
\big[ [x_1,x_2,x_3] , x_4,x_5 \big] = \big [ x_1,[x_4,x_3,x_2],x_5\big] =  \big[ x_1,x_2,[x_3 , x_4,x_5] \big]\, ,
\end{equation*}
for all $x_i \in S$.   A semiheap is said to be an \emph{abelian semiheap} if $[x_1,x_2, x_3] =  [x_3, x_2, x_1]$ for all $x_i \in S$.  A semiheap is a \emph{heap} when all its elements are \emph{biunitary}, i.e., $[y,x,x] =y=[x,x,y]$, for all $y$ and $x \in S$. We remark that a general semiheap is not associated (up to isomorphism) with a group, this is a particular property of heaps. However, every semiheap can be embedded in an involuted semigroup. For more details about heaps and related structures the reader should consult Hollings \& Lawson \cite{Hollings:2017} who discuss Wagner’s original development of the theory.\par 
In this paper, we extend the definition of a semiheap to the category of smooth manifolds, i.e., we will study Lie semiheaps as a natural generalisation of Lie groups. As a side remark, the notion of the ``Lie category'' goes back to Ehresmann (1959) and Grothendick (1960/61). We then proceed to define and study semiheap bundles, which are akin to principal bundles in the ternary setting.  We show that principal bundles provide a class of semiheap bundles.  The main complications as compared with the standard situation with Lie groups is that we do not have an identity element, and the left/right translations are not diffeomorphisms. While left-invariant vector fields make sense, and we will explore this, one cannot construct a finite dimensional Lie algebra associated with any Lie semiheap. Thus, many of the statements found in standard Lie theory are somewhat obscured.  None-the-less, there is still a rich  and potentially useful theory here to be uncovered. 

\subsection{Motivation}
The main motivation for this work comes from the recent renewed interest in ternary operations such as trusses, which are ``ring-like" structures with a ternary ``addition" (see \cite{Brzezinski:2018, Brzezinski:2019, Brzezinski:2020, Brzezinski:2019b}). In particular, the question of defining geometry based on ternary operations has arisen. With this in mind, the first steps in this direction was to find geometric examples of heaps and semiheaps, and from there attempt to replace binary operations with ternary ones. A very recent observation by Breaz, Brzezi\'{n}ski, Rybo{\l}owicz \&  Saracco, (see \cite{Breaz:2022}), is that affine spaces can be defined without reference to vector spaces, using  two ternary operations; one entirely on the set and  the other representing an action of the field of scalars. This reformulation is consistent with the idea that we should not fix an observer or gauge when formulating physics. Here, selecting the zero vector is like picking an observer in physics. Even in high school physics, energy, voltage and position are not `absolutely measurable', but one needs to fix a zero point and measure things relative to this chosen point - mathematically one should be thinking in terms of torsors.  A gauge or frame-independent formulation of analytical dynamics requires affine bundles and, because of this, Grabowska, Grabowski and Urba\'{n}ski (see \cite{Grabowska:2003}) defined Lie brackets on sections of affine bundles -  this too has been reformulated by  Brzezi\'{n}ski using ternary operations (see \cite{Brzezinski:2022}). Semiheaps and ternary algebras have been applied to quantum mechanics (see \cite{Bruce:2022,Kerner:2008,Kerner:2018}). Heaps also appear in knot theory in various guises (see for example \cite{Saito:2021}).\par 
We must of course mention torsors in algebraic geometry, which are generalisations of principal bundles in algebraic topology. Torsors in this setting  can also be viewed as generalising  Galois extension as found in abstract algebra. There are deep links here with various branches of mathematical physics, including deformation quantisation (see, for example, \cite{Kontsevich:1999}).  It is also remarkable that heaps and trusses appear behind elliptical curves and their endomorphisms (see \cite{Brzezinski:2022b}.)\par  
Another source of motivation comes from the applications of nonassociative algebras in geometry and geometric mechanics, for example smooth loopoids of Grabowski \& Ravanpak (see \cite{Grabowski:2016,Grabowski:2021}).  Ternary operations are, by definition, nonassociative. 
\subsection*{A Historical Note}  
After the first draft of this paper, the author was notified that semiheap bundles were considered by Konstantinova in 1978, albeit in Russian, see \cite{Konstantinova:1978}.
\section{Lie Semiheaps}
\subsection{Semiheaps}
In this subsection we review, and slightly reformulate the notion of a ternary multiplication and a semiheap. Nothing in this subsection is new. Our main reference is Hollings \& Lawson \cite{Hollings:2017}. Let $S \in \catname{Set}$ be a  set (possibly empty). We define $S^{(n)} :=  S \times S \times \cdots \times S$ where there are $n$-factors, and similar for set theoretical maps.  The group $\sS_3$ acts on $S^{(3)}$ in a canonical way. Specifically, and vital for our later needs, 
\begin{align*}
  s_{13}: & ~  S^{(3)} \longrightarrow S^{(3)}\\
& (x,y,z) \longmapsto (z,y,x)\,.
\end{align*} 
 A ternary multiplication/product we write as
  \begin{align*}
  \mu : & ~  S^{(3)} \longrightarrow S \\
& (x,y,z) \longmapsto [x,y,z]\,.
\end{align*} 
When convenient, we will write $\mu(x,y,z)$ for the ternary multiplication.
\begin{definition}\label{def:SemiHeap}
A set (possibly empty) equipped with a ternary product $(S, \mu)$ is a \emph{semiheap} if the product is para-associative, i.e., the following diagram is commutative:
\begin{center}
\leavevmode
\begin{xy}
(0,15)*+{S^{(2)}\times S^{(3)}}="a"; (50,15)*+{S \times S^{(3)}\times S}="b"; (100,15)*+{S^{(3)} \times S^{(2)}}="c";%
(0,-10)*+{S^{(3)}}="d"; (50,-10)*+{S^{(3)}}="e"; (100,-10)*+{S^{(3)}}="f";%
(50, -25)*+{S}="g";
{   \ar@{=} "a";"b"}; {\ar@{=} "b";"c"} ;%
{\ar "a";"d"}?*!/^10mm/{\Id^{(2)}\times \mu }; {\ar "b";"e"}?*!/^15mm/{\Id\times(\mu \circ s_{13})\times \Id }; {\ar "c";"f"}?*!/^10mm/{ \mu \times \Id^{(2)}};%
{\ar "d";"g"}?*!/^3mm/{\mu};{\ar "e";"g"}?*!/^3mm/{\mu}; {\ar "f";"g"}?*!/_3mm/{\mu};%
\end{xy}
\end{center}
\end{definition}
The para-associative property concretely means
$$[x_1, x_2 , [x_3, x_4, x_5]] = [x_1, [x_4, x_3, x_2], x_5] = [[x_1, x_2, x_3], x_4, x_5]\, ,$$
for all $x_i \in S$. A semiheap is a \emph{heap} if every element of the semiheap is biunitary, i.e., $y = [x,x,y] = [y,x,x]$.
 \begin{example}
 Let $M$ be a smooth manifold, then the set of smooth functions on $M$ can be considered as a semiheap with the ternary operation being $[f_1, f_2, f_3] :=  f_1 f_2 f_3$. As the algebra of smooth functions is commutative, the para-associativity of the ternary product is clear.
 \end{example}
\begin{example}
Let $M$ be a smooth manifold, then the set of nowhere vanishing smooth functions on $M$ can be considered as a heap with the ternary operation being $[f_1, f_2, f_3] :=  f_1 (f_2)^{-1} f_3$.
\end{example}
\begin{example}
Consider the group of diffeomorphisms between two smooth manifolds $\Diff(M,N)$. Then we have a natural heap structure given by
$$[\phi_1, \phi_2, \phi_3] := \phi_1 \circ \phi_2^{-1}\circ \phi_3\,,$$
for all $\phi_i \in\Diff(M,N)$.
\end{example}
\begin{example}
Let $(M, g)$ be a Riemannian manifold. Then on the set of vector fields on $M$ we can define a ternary product as
$$[X,Y,Z] := X \, g(Y,Z)\,.$$
The fact that $g(W,X)g(Y,Z) = g(W, Xg(Y,Z)) = g(Y, g(X, W)Z)$ for arbitrary vector fields implies that the ternary product is para-associative, i.e., we have a semiheap on the set  of vector fields on a Riemannian manifold. There is also a heap structure on any vector space (or more generally an affine space) given by $\{X, Y, Z \} :=  X -Y +Z$, which is independent of any Riemannian structure or similar. 
\end{example}
\begin{remark}
The analogue construction for a symplectic manifold $(M, \omega)$  does not quite give a semiheap as $\omega(W, X) = - \omega(X,W)$ and so there is an extra sign present when examining the para-associativity.
\end{remark}
\begin{example}
Let $\mathcal{C}(E)$ be the affine space of linear connections on a vector bundle $\pi : E \rightarrow M$. Then $\mathcal{C}(E)$ is canonically a heap with the ternary operation being $[\nabla_1, \nabla_2, \nabla_3] :=  \nabla_1 - \nabla_2 + \nabla_3$.
\end{example}
\begin{definition}
Let $(S', \mu')$ be a semiheap and let $S \subseteq S'$ be a subset. Then $S$ is a \emph{subsemiheap} of $S'$ if it is closed with respect to the ternary product. In other words, $(S, \mu := \mu'|_{S})$ is a semiheap.
\end{definition}
\begin{definition}
Let $(S, \mu)$ and $(S', \mu')$ be semiheaps. A \emph{homomorphism of semiheaps} is a map $\phi : S \rightarrow S' $ that satisfies
 $$\phi \circ \mu = \mu' \circ \phi^{(3)}\,.$$
\end{definition} 
Concretely, a map is a homomorphism of semiheaps if 
$$\phi[x_1, x_2, x_3] =  [\phi(x_1),  \phi(x_2), \phi(x_3)]'\,,$$
for all $x_i \in S$. We thus obtain the category of semiheaps, which we denote as $\catname{SHeap}$.
\begin{example}
If $(S, \mu)$ is a subsemiheap of $(S', \mu')$, then the inclusion map $\iota :  S \longhookrightarrow S'$ is a homomorphism of semiheaps.
\end{example}
\begin{example}
Let $\phi:  (M, g) \rightarrow (M', g')$ be an isometry of Riemannian manifolds. We observe that 
$$\phi_*\big(X g(Y,Z) \big) = \phi_*(X)g(Y,Z)= \phi_*(X)g'(\phi_*(Y),\phi_*(Z))\,,$$
and so an isometry induces a homomorphism of the associated semiheaps.
\end{example}
\begin{definition}
Let $(S, \mu)$ and $(S', \mu')$ be semiheaps and $\phi : (S, \mu) \rightarrow (S', \mu')$ be a homomorphism of semiheaps. Then the \emph{homomorphic image} of $\phi$ is the set 
$$\phi(S) = \{ a \in S' ~~ |~~ \exists\, x \in S~ \textnormal{such that}~ a = \phi(x)\} \subseteq S'\,.$$
\end{definition}  
A standard argument from universal algebra, that is sets with operations, establishes the following.
\begin{proposition}
Let $\phi : (S, \mu) \rightarrow (S', \mu')$ be a homomorphism of semiheaps. The homomorphic image $\phi(S) \subseteq S'$ is a semiheap.
\end{proposition}
\begin{definition}\label{def:RLActions}
Let $(S, \mu)$ be a semiheap and fix a pair $(x_1, x_2) \in S^{(2)}$. The map 
\begin{align*}
R_{x_1 x_2}: &~ S \longrightarrow S\\
& x \mapsto [x, x_1, x_2]
\end{align*}
is a \emph{right translation}. Similarly, the map
\begin{align*}
L_{x_1 x_2}: &~ S \longrightarrow S\\
& x \mapsto [ x_1, x_2, x]
\end{align*}
is a \emph{left translation}.
\end{definition}
We will denote the set of right translations of a semiheap by $R(S)$ and the set of left translations of a semiheap by $L(S)$.
\begin{proposition}\label{prop:TransSemi}
Let $(S, \mu)$ be a semiheap. Then the sets $R(S)$ and $L(S)$ of right and left translations, respectively, are semigroups. 
\end{proposition}
\begin{proof}
We will consider right translations as these are of importance for bundles. The case of left translations follows in more-or-less the same way.
\begin{enumerate}
\item The composition of two right translations is a right translation:
\begin{align*}
R_{x_3 x_4}\circ R_{x_1 x_2}(-) &= R_{x_3 x_4}([- , x_1, x_2])\\
								& = [[- , x_1, x_2 ],x_3,x_4]\\
								& = [-, x_1, [x_2, x_3, x_4]]\\
								&= R_{x_1[x_2, x_3, x_4] }(-)\,.
\end{align*}
\item Associativity of the composition: 
\begin{align*}
\big ( R_{x_5 x_6} \circ R_{x_3 x_4} \big)\circ R_{x_1 x_2} &= R_{x_3 [x_4,x_5, x_6]} \circ R_{x_1 x_2}\\
& = R_{x_1[x_2, x_3, [x_4,x_5,x_6 ]]}\\
& = R_{x_1[x_2, x_3, x_4],x_5,x_6 ]}\\
&= R_{x_5 x_6} \circ R_{x_1 [x_2 x_3 x_4]}\\
&=  R_{x_5 x_6} \circ\big ( R_{x_3 x_4} \circ R_{x_1 x_2}\big)\,.
\end{align*}
\end{enumerate}  
\end{proof}
For left translations, it can directly be shown that $L_{x_1 x_2} \circ L_{x_3 x_4} = L_{[x_1 x_2, x_3] x_4}$ and that we have associativity of the composition of left translations. 
\begin{definition}
Let $(S, \mu)$ be a semiheap and fix a pair $(x_1, x_2) \in S^{(2)}$. The map 
\begin{align*}
C_{x_1 x_2}: &~ S \longrightarrow S\\
& x \mapsto [x_1, x, x_2]
\end{align*}
is a \emph{centric translation}.
\end{definition}
\begin{warning}
Due to para-associativity, we do not have a semigroup structure on the set of centric translations. The composition of centric translations is not a centric translation.
\end{warning}
\begin{proposition}
Left and right translations on a semiheap $S$ commute, i.e., for all $x_i \in S$, 
$$L_{x_1 x_2} \circ R_{x_3  x_4} =  R_{x_3  x_4}\circ L_{x_1 x_2} \,.$$
\end{proposition}
\begin{proof}
Let $x_i$ and $x \in S$. Then directly
$$L_{x_1 x_2}\big(R_{x_3 x_4}(x) \big) =  [x_1, x_2 ,[x, x_3, x_4]] = [[x_1, x_2, x], x_3, x_4] =  R_{x_3  x_4}\big( L_{x_1 x_2}(x) \big)\,,$$ 
using the para-associative property.
\end{proof}
\subsection{Lie Semiheaps}
Recall that a Lie group is a group object in the category of smooth manifolds, i.e., a smooth manifold equipped with three smooth maps, the unit, inverse and multiplication maps, that satisfy the standard axioms of a group. The reader may consult Mac Lane \cite[pages 75-76]{MacLane:1988} for details of group objects in categories.  A Lie semiheap is a semiheap object in the category of smooth manifolds. That is, a smooth manifold together with a ternary operation that satisfies the axioms of a semiheap (see Definition \ref{def:SemiHeap}).  More formally,  we make the natural definition  of a Lie semiheap. 
\begin{definition}
A \emph{Lie semiheap} is a semiheap object in the category of real, finite dimensional, Hausdorff and second countable smooth manifolds. In particular, the map $\mu : S^{(3) } \rightarrow S$ is a smooth map. A \emph{Lie semiheap homomorphism} is a smooth map $\psi : S \rightarrow S'$ that is also a semiheap homomorphism.
\end{definition}
\begin{remark}
The notion of a \emph{topological semiheap} is evident as a semiheap object in the category of topological spaces and so the ternary product is continuous.   We will restrict attention to the smooth case in this paper.  Moreover, \emph{complex Lie semiheaps} can similarly be defined as semiheap objects in the category of complex manifolds, so in particular, the ternary product is holomorphic.  One can also consider Lie semiheaps in the category of supermanifolds - the functor of points is expected to be a useful concept in this context. We will only consider real and finite dimensional manifolds here.
\end{remark}
The resulting category of Lie semiheaps we denote as $\catname{LieSHp}$. Within this category is the full subcategory of Lie heaps, which we denote as $\catname{LieHp}$. Specifically, the forgetful functor $\textrm{F}: \catname{LieHp} \rightarrow \catname{LieSHp}$, which forgets the biunitary property of all elements, is is full, faithful, and injective. Generically, we will not distinguish Lie heaps and Lie semiheaps, but rather consider Lie heaps as particular Lie semiheaps. Furthermore,  there are other full subcategories  $\catname{AbLieSHp}$ of abelian Lie semiheaps, and $\catname{AbLieHp}$ of abelian Lie heaps.
\begin{remark}
Recall that the category of Lie groups is, similarly, a full subcategory of the category of Lie semigroups. The forgetful functor in this case forgets the identity element and the inverse map. 
\end{remark}
\begin{example}
If $(S, \eta)$ is a Lie semiheap, then $(S, \eta^{\textrm{op}})$ is also a Lie semiheap where we define $[x,y,z]^{\textrm{op}} := [z,y,x]$. Clearly, a Lie semiheap is abelian if and only if $\eta^{\textrm{op}} = \eta$.
\end{example}
\begin{example}
A singleton $\{\star\}$ considered as a zero-dimensional smooth manifold has a unique heap operation $\{\star\} \times \{\star\} \times \{\star\} \rightarrow \{\star\}$.  This ternary product is smooth and so we have the \emph{trivial Lie (semi)heap}.
\end{example}  
\begin{example}
The empty set $\emptyset$, can be considered (conventionally) as a smooth manifold and comes with a unique heap operation $\emptyset \times \emptyset \times \emptyset \rightarrow \emptyset$. This ternary product is by definition smooth and so we have the \emph{empty Lie (semi)heap}. 
\end{example}
\begin{remark} 
Universal statements hold for the empty set, however existence statements are false. Thus, the empty set is not a (Lie) group as we require the existence of an identity element. 
\end{remark}
From the well-known observations about the category of smooth manifold and semiheaps, the following is evident (see \cite[page 20]{MacLane:1988} for the notion of initial and terminal objects).
\begin{proposition}
The empty semiheap is the initial object on the category of Lie semiheaps. The trivial Lie semiheap is the terminal object in the category of semiheaps.
\end{proposition}
We will need a slightly modified notion of a Lie semiheap in which a distinguished point is identified. 
\begin{definition}
A \emph{pointed Lie semiheap} is a triple $(S, \mu, \textrm{pt})$, such that $(S, \mu)$ is a Lie semiheap, and $\textrm{pt}\in S$ is a distinguished point. A \emph{pointed Lie semiheap homomorphism} $\phi :  (S, \mu, \textrm{pt}) \rightarrow (S', \mu', \textrm{pt}')$ is a Lie semiheap homomorphism such that $\phi(\mathrm{pt}) = \mathrm{pt}'$. The resulting category of pointed Lie semiheaps will be denoted as $\catname{LieShp}_*$.
\end{definition}
There is the obvious forgetful functor $\catname{LieShp}_* \rightarrow \catname{LieShp}$ in which the distinguished point is forgotten.
\begin{example}
Any Lie group can be considered as a pointed Lie semiheap (in fact a pointed Lie heap). In particular, if $G$ is a Lie group, then we define $[g_1, g_2, g_3] := g_1 g_2^{-1}g_3$. As a map $\mu : G^{(3)}\rightarrow G$, it is clear that the ternary multiplication is smooth as, by definition, the group product and inversion are smooth. If $\psi : G \rightarrow G'$ is  Lie group homomorphism, then it is also a Lie semiheap homomorphism. This is easily seen from the properties of a Lie group homomorphism, i.e., $\psi [g_1, g_2, g_3] = \psi(g_1 g_2^{-1}g_3) = \psi(g_1) \psi(g_2)^{-1}\psi(g_3) = [\psi(g_1), \psi(g_2),\psi(g_3)]'$. The distinguished point is the identity element $e \in G$. Furthermore, for any Lie group homomorphism we have that $\psi(e) = e'$.
\end{example}
The previous example shows that we have a functor from the category of Lie groups to the category of pointed Lie semiheaps. More formally, we have the following definition - already a well known result in the setting of groups and heaps.
\begin{definition}\label{def:HeapFunc}
The \emph{heapification functor} is the functor
$$\mathcal{H} : ~ \catname{LieGrp} \longrightarrow \catname{LieShp}_*\,,$$
that on objects acts as
$$(G, \textrm{m}, \textrm{i})\longmapsto (G, \mu)\,,$$
where $\mu(g_1,g_2, g_3) = [g_1, g_2, g_3] :=  g_1 g_2^{-1}g_3$, the distinguished point is the identity element $e\in G$, and on morphisms $\psi : G \rightarrow G'$, acts as $\psi^\mathcal{H}:=  \psi$.
\end{definition}
As a matter of notation, we will set $S_G := \mathcal{H}(G)= (G, \mu, e)$ to denote the pointed Lie (semi)heap associated with a Lie group $G$.\par 
We remind the reader that a functor is full if it is surjective on the hom sets and is faithful if it is injective on the hom sets. A functor is said to be fully faithful if is full and faithful, i.e., is a bijection between the hom sets. 
\begin{proposition}\label{prop:HeapFunFulFaith}
The heapification functor is fully faithful.
\end{proposition}
\begin{proof}
As the heapificaion functor does not change a given Lie group homomorphism, it is just considered as being in a different category, it is obviously faithful.  The only thing to check is that any $\bar{\psi} : S_G \rightarrow S_{G'}$ is also a group homomorphism. Note that, by definition we have $\bar{\psi}(e) = e'$. As a homomorphism of semiheaps, it must be the case that
 $$ \bar{\psi}(g_1 g_2^{-1} g_3) = \bar{\psi}(g_1)\big(\bar{\psi}(g_2)\big)^{-1}\bar{\psi}(g_3)\,.$$
Setting $g_2 =  e$ in the above gives $\bar{\psi}(g_1  g_3) =\bar{\psi}(g_1)(\bar{\psi}(e))^{-1}\bar{\psi}(g_3)= \bar{\psi}(g_1)e'\bar{\psi}(g_3) =\bar{\psi}(g_1)\bar{\psi}(g_3)$. This result implies that we have a group homomorphism, and so the heapification functor is full. 
\end{proof}
\begin{remark} 
If we consider the category of Lie semiheaps, rather than the category of pointed Lie semiheaps, then the resulting heapification functor is \emph{not} fully faithful as there is no reason why $\bar{\psi}(e) = e'$ for an arbitrary homomorphism of Lie semiheaps.  However, it may still be convenient to consider the range of the heapification functor as $\catname{LieShp}$.
\end{remark}
\begin{example}
$\R^n$ with its standard topology and smooth structure is an abelian Lie group with respect to addition. The associated heap operation is thus $[\mathrm{x}, \mathrm{y}, \mathrm{z}] := \mathrm{x} -  \mathrm{y} + \mathrm{z}$. The distinguished point is the zero element $\mathrm{0}\in \R^n$. 
\end{example}
\begin{example}
$\R^\times$, the set of non-zero real numbers, with its standard topology and smooth structure is an abelian Lie group with respect to multiplication. The associated heap operation is thus $[\mathrm{x}, \mathrm{y}, \mathrm{z}] := \mathrm{x} \mathrm{y}^{-1} \mathrm{z}$. The distinguished element is $1 \in \R^\times$. 
\end{example}
\begin{example}
Combining the two previous examples, the exponential map $(\R, + ) \rightarrow (\R^\times, \cdot)$ induces a homomorphism of the associated pointed Lie semiheaps. Explicitly, 
$$\rme^{\mathrm{x} -  \mathrm{y} + \mathrm{z}} = \rme^{\mathrm{x}} \rme^{ -  \mathrm{y}} \rme^{ \mathrm{z}} = \rme^{\mathrm{x}} (\rme^{  \mathrm{y}})^{-1} \rme^{ \mathrm{z}}\,. $$
\end{example}
We note that the target category of the heapification functor can be restricted to the category of pointed Lie heaps,  which we denote as $\catname{LieHp}_*$. We then have a functor going in the other direction, which is again, well known in the context of heaps and groups.
\begin{definition}\label{def:GroupFun}
The \emph{groupification functor} is the functor
$$\mathcal{G} : \catname{LieHp}_* \longrightarrow \catname{LieGrp}\,,$$
which acts on objects as
$$(H, \mu, e) \longmapsto (H, m, i)\,,$$
where $m(x_1, x_2) := [x_1, e, x_2]$, and $i(x) = x^{-1} := [e,x,e]$. Furthermore, the group identity element is $e \in H$.  On morphisms
$\phi : (H, \mu, e) \rightarrow (H', \mu', e')$ the groupification functor acts as $\phi^{\mathcal{G}} := \phi$.
\end{definition}
Note that as the ternary product is by definition smooth, the corresponding group structure is also smooth and the target category of the groupification functor is the category of Lie groups.
\begin{proposition}
The groupification functor is fully faithful.
\end{proposition}
\begin{proof}
As the groupification functor does not change a given pointed Lie heap homomorphism, it is just considered as being in a different category, it is obviously faithful. Fullness following using the direct observation that $x y^{-1}z = [x,y,z]$. Then applying an arbitrary Lie group homomorphism $\bar{\phi}:  \mathcal{G}(H) \rightarrow \mathcal{G}(H')$ shows that 
$$ \bar{\phi}([x,y,z]_H) = \bar{\phi}(x y^{-1}z) = \bar{\phi}(x) \bar{\phi}(y)^{-1}\bar{\phi}(z) = [\bar{\phi}(x),\bar{\phi}(y),\bar{\phi}(z)]_{H'}\,,$$
which established the desired result.
\end{proof}
Recall that two categories $\catname{C}$ and $\catname{D}$ are isomorphic if there exists two functors $\mathrm{F} : \catname{C} \rightarrow \catname{D}$ and $\mathrm{G} : \catname{D} \rightarrow \catname{C}$ such that $\mathrm{F}\mathrm{G} = \Id_D$ and $\mathrm{G}\mathrm{F} = \Id_C$. That is, there is a one-to-one correspondence between objects and morphisms. As it can easily and directly be shown that the on objects the heapification and groupification functors are mutual inverses - a fact well known in the algebraic setting - we have the following theorem.
\begin{theorem}\label{trm:GrpHeapIso}
There is an isomorphisms of categories between $\catname{LieGrp}$ and $\catname{LieHp}_*$.
\end{theorem}
\begin{proposition}\label{prop:InducedSemiHeap}
Let $M$ be a smooth manifold, $(S, \mu)$ be a Lie semiheap, and $\phi, \psi :  M \rightarrow S$ be diffeomorphisms. Then $M$ inherits two Lie semiheap structures, $\mu_\phi$ and $\mu_\psi$, that are canonically isomorphic. 
\end{proposition}
\begin{proof}
The ternary structures inherited are the obvious ones, i.e., we set 
\begin{align*}
[m_1,m_2,m_3]_\phi := \phi^{-1} [\phi(m_1), \phi(m_2), \phi(m_3)]\, , &&[m_1,m_2,m_3]_\psi :=  \psi^{-1}[\psi(m_1), \psi(m_2), \psi(m_3)] \, .
\end{align*}
We first need to show that these ternary structures are para-associative. We chose to study the structure associated with $\phi$, but, of course, the case of $\psi$ follows. Using the para-associativity of the ternary multiplication on $S$ and the fact that $\phi \phi^{-1} = \Id$, we observe that
 \begin{subequations}
\begin{align}
\nonumber [[m_1, m_2, m_3]_\phi, m_4, m_5]_\phi & = \phi^{-1}[\phi \phi^{-1}[\phi(m_1), \phi(m_2), \phi(m_3)], \phi(m_4), \phi(m_5)]\\
&= \phi^{-1}[\phi(m_1), \phi \phi^{-1}[\phi(m_4), \phi(m_3), \phi(m_2), \phi(m_5)]]\label{prf:a}\\
& = \phi^{-1}[\phi(m_1), \phi(m_2), \phi \phi^{-1}[\phi(m_3), \phi(m_4), \phi(m_5)]]\,. \label{prf:b}
\end{align}
\end{subequations}
We then note that \eqref{prf:a} is identical to $[m_1, [m_4,m_3, m_4]_\phi, m_5]_\phi$, and that \eqref{prf:b} is identical to $[m_1, m_2, [m_3, m_4, m_5]_\phi]_\phi$. Thus, we have para-associativity of the induced ternary operations.  Note that $\phi$ is a Lie semiheap homomorphism from the induced structure to the one on $S$.\par 
Next we need to show that $\psi^{-1}\circ \phi : M \rightarrow M$ is a Lie semiheap homomorphism between the two induced structures. Clearly, as a composition on diffeomorphisms is itself a diffeomorphism, we will have an isomorphism of Lie semiheaps.  Directly,
\begin{align*}
\psi^{-1}\phi[m_1, m_2, m_3]_\phi & = \psi^{-1}\phi \phi^{-1}[\phi(m_1), \phi(m_2), \phi(m_3)]\\
& = \psi^{-1}[\phi(m_1), \phi(m_2), \phi(m_3)]\\
& = \psi^{-1}[\psi \psi^{-1}\phi(m_1), \psi \psi^{-1}\phi(m_2),\psi \psi^{-1} \phi(m_3)]\\
& = [\psi^{-1}\phi(m_1), \psi^{-1}\phi(m_2), \psi^{-1}\phi(m_3)]_\psi\,,
\end{align*}
as required. A similar statement holds for $\phi^{-1}\circ \psi$.
\end{proof}
We can modify Proposition \ref{prop:InducedSemiHeap} by considering a pointed Lie semiheap $(S, \mu, \mathrm{pt})$. If we set $\phi^{-1}( \mathrm{pt})=m$ and $\psi^{-1}( \mathrm{pt})=n$, then we have a diffeomorphism of pointed manifolds $\psi^{-1}\circ \phi :  (M, m)\rightarrow (M, n)$.  The following proposition is thus evident.
\begin{proposition}\label{prop:InducedPntSemiHeap}
Let $M$ be a smooth manifold, $(S, \mu, \textrm{pt})$ be a pointed Lie semiheap, and $\phi, \psi :  M \rightarrow S$ be diffeomorphisms such that $\phi^{-1}( \mathrm{pt})=m$ and $\psi^{-1}( \mathrm{pt})=n$. Then the two inherited pointed Lie semiheaps $(M, \mu_\phi, m)$ and $(M, \mu_\psi, n)$ are canonically isomorphic.  
\end{proposition}
Let $(S,\mu)$ and $(S', \mu')$ be Lie semiheaps. Then $S\times S'$ is, of course, a smooth manifold. A ternary product on the Cartesian product can be defined as
$$[(x_1, y_1), (x_2, y_2), (x_3, y_3)] := \big ( [x_1, x_2, x_3], [y_1, y_2, y_3] \big)$$
 Clearly, the ternary product is para-associative and smooth. Thus, the Cartesian product of Lie semiheaps is again a Lie semiheap. It remains to argue that the Cartesian product is a categorical product. 
 \begin{proposition}
In the category of Lie semiheaps, $\catname{LieSHp}$, the Cartesian product is a categorical product.
\end{proposition}
\begin{proof}
The preceding discussion shows that the Cartesian product of two Lie semiheaps is again a semiheap. We only have to demonstrate the universal property. Let $(S,\mu)$ and $(S', \mu')$ be Lie semiheaps. We then define the projection maps (which are clearly homomorphisms of Lie semiheaps)
\begin{align*}
\pi_S : S\times S' \rightarrow S\,, && \pi_{S'} : S\times S' \rightarrow S'\,.
\end{align*}
Let $(T, \nu)$ be any Lie semiheap and consider the pair of Lie semiheap homomorphisms  
\begin{align*}
\phi_S : T \rightarrow S\,, && \phi_{S'} : T \rightarrow S'\,.
\end{align*}
The universal property is that given the above homomorphisms, there exists a unique Lie semiheap homomorphism $\phi :  T \rightarrow S\times S' $, such that the following diagram is commutative:
 \begin{center}
\leavevmode
\begin{xy}
(0,20)*+{T}="a"; (30,20)*+{S}="b";%
(0,0)*+{S'}="c"; (30,0)*+{S \times S'}="d";%
{\ar "a";"b"}?*!/_3mm/{\phi_S };%
{\ar "a";"c"}?*!/^3mm/{\phi_{S'}};{\ar@{-->} "a";"d"}?*!/^4mm/{\phi};%
{\ar "d";"b"}?*!/^3mm/{\pi_S};{\ar "d";"c"}?*!/_5mm/{\pi_{S'}};%
\end{xy}
\end{center} 
We claim that the required map is $\phi(-) := \big(\phi_S(-), \phi_{S'}(-) \big)$.  Clearly this map is smooth and renders the above diagram commutative. It is easy to check that this map is a Lie semiheap homomorphism.
\end{proof}
Recall that the tangent functor  (see \cite[Chapter I]{Kolar:1993}) is a functor from the category of smooth manifolds to the category of smooth manifolds that 
\begin{enumerate}
\item on objects, sends $M$ to its tangent bundle $\sT M$, and 
\item on morphisms, $\psi:  M \rightarrow N$ gets sent to $\sT \psi :  \sT M \rightarrow \sT N$.
\end{enumerate}
A fundamental property of the tangent functor is  that it preserves products, i.e.,  $\sT (M \times N) \cong \sT M \times \sT N$, and given $\psi : M \rightarrow M'$ and $\chi : N \rightarrow N'$, $\sT(\psi \times \chi) \cong \sT \psi \times \sT \chi $.
\begin{proposition}
Let $(S, \mu)$ be a Lie semiheap. Then $(\sT S, \sT \mu)$ is also a Lie semiheap.
\end{proposition}
\begin{proof}
We need to check that $\sT \mu  : (\sT S)^{(3)} \rightarrow  \sT S$ is para-associative. We start with the para-associative property of $(S, \mu)$ 
$$\mu \circ(\Id^{(2)}\times \mu) = \mu \circ( \Id\times(\mu \circ s_{13})\times \Id) = \mu \circ  (\mu \times \Id^{(2)})\,,$$
and apply the tangent functor.  Using the properties of the tangent functor, and via minor abuse of notation, we observe that
$$\sT \mu \circ(\Id^{(2)}\times \sT \mu) = \sT \mu \circ( \Id\times(\sT\mu \circ s_{13})\times \Id) = \sT \mu \circ  (\sT \mu \times \Id^{(2)})\,,$$
thus we have the para-associative property. 
\end{proof}
\begin{definition}
Let $(S, \mu)$ be a Lie semiheap. The Lie semiheap $(\sT S, \sT \mu)$ will be referred to as the \emph{tangent Lie semiheap} of $(S, \mu)$.
\end{definition}
If we have a pointed Lie semiheap $(S, \mu, \mathrm{pt})$, then the tangent functor produces the pointed Lie semiheap $(\sT S, \sT \mu, \mathrm{0}_{\textrm{pt}})$.
\subsection{Left-invariant Vector Fields}
 The notion of a left-invariant vector field directly generalises to the setting of Lie semiheaps. 
\begin{definition}
Let $(S, \mu)$ be a Lie semiheap. Then a vector field $X \in \Vect(S)$ is said to be a \emph{left-invariant vector field} if
$$l^* \circ X =  X \circ l^*\,,$$
for all $l \in L(S)$. The vector space of all left-invariant vector fields is denoted as $\Vect_L(S)$.
\end{definition}
\begin{proposition}
The space of left-invariant vector fields $\Vect_L(S)$ on a Lie semiheap $(S, \mu)$ is a Lie subalgebra of the Lie algebra of vector fields. 
\end{proposition}
\begin{proof} 
We only need to show that the space is closed under the standard commutator bracket. Explicitly, assuming $X,Y \in \Vect_L(S)$ we have
\begin{align*}
l^*\circ [X,Y] &= l^*\circ(X\circ Y - Y \circ X) = l^*\circ X\circ Y - l^*\circ Y \circ X =  X\circ Y  \circ l^* - Y \circ X \circ l^*\\
&= [X,Y]\circ l^*\,.
\end{align*}
\end{proof}
The left-invariant condition can be expressed as
$$(\rmd L_{xy})_z X_z =  X_{[x,y,z]}\,$$
for all $x,y$ and $z \in S$. For fixed $(x,y) \in S^{(2)}$, the derivative map is understood as the linear map
$$\rmd L_{xy}:  \sT _z S \longrightarrow \sT_{[x,y,z]}S\,.$$
In order to further discuss properties of left-invariant vector fields, we will specialise to a particularly nice class of pointed Lie semiheaps. 
\begin{definition}
A pointed Lie semiheap $(S, \mu, \textrm{pt}=: x_0)$ is said to be a \emph{biunital Lie semiheap} if the distinguished point is biunital, i.e., $[x,x_0,x_0] = x = [x_0, x_0, x]$ for all points $x \in S$.  \emph{Homomorphisms of biunital Lie semiheaps} are homomorphisms of pointed Lie semiheaps. The resulting category we denote as $\catname{BULieShp}$.
\end{definition}
\begin{proposition}
The set of left-translations $L(S)$ on a biunital Lie semiheap is a monoid.
\end{proposition}
\begin{proof}
Via Proposition \ref{prop:TransSemi}, we know that $L(S)$ is a semigroup. We only need to show the existence of the identity. We claim that $L_{x_0 x_0}= \Id_{L(S)}$.  Directly,
\begin{align*}
& L_{x_0 x_0}\circ L_{yz} = L_{[x_0,x_0, y]z} = L_{xy}\,,\\ 
& L_{yz}\circ L_{x_0x_0} = L_{[y, z, x_0] x_0} = [[y,z,x_0],x_0,-]=  [y,z , [x_0,x_0,-]] = L_{yz}\,.
\end{align*}
\end{proof}
\begin{remark}
The same statement holds for right-translations on a biunital Lie semiheap.
\end{remark}
\begin{proposition}\label{prop:BiLieReach}
Let $(S, \mu, x_0)$ be a biunital Lie semiheap.  Then any point $x \in S$ can be reached from $x_0$ via a left-translation. 
\end{proposition}
\begin{proof}
Observe that for any $x \in S$
$$L_{x x_0}(x_0) = [x, x_0, x_0] =  x\,.$$
\end{proof}
\begin{remark}
The definition of a left-invariant functions is clear, i.e., $l^*f =f$ for all $l \in L(S)$. This condition implies, for biunital Lie semiheaps, that 
$$f(x) = f([x, x_0, x_0]) =  f(x_0)\,,$$
and so left-invariant functions are constants. This conclusion is not evident on general Lie semiheaps.
\end{remark}
\medskip
Note that $(\rmd L_{x x_0})_{x_0} :  \sT_{x_0}S \longrightarrow \sT_x S$ is not, in general, an isomorphism as the underlying left-translation is not a diffeomorphism. In other words, $(\rmd L_{x x_0})_{x_0}$ will, in general, have a non-trivial kernel. This complicates the question of the existence and dimension of the the space of left-invariant vector fields.
\begin{proposition}
Let $(S, \mu, x_0)$ be a biunital Lie semiheap. The map $V : S \rightarrow \sT S$ given by $x \mapsto (\rmd L_{x x_0})_{x_0}v$ for a given (non-zero) $v \in \sT_{x_0}S$, defines a smooth vector field on $S$ that is left-invariant, such that $V(x_0) = v$.
\end{proposition}
\begin{proof}
 From Proposition \ref{prop:BiLieReach}, we know that any point in $S$ can be reached from $x_0$ via a left-translation and so the map $V$ is well defined. It is clear that $V$ can be considered as a vector field, a priori, which may not be smooth.  To see smoothness, obverse that that $V$ being smooth is equivalent to $V(f)$ being smooth for all $f \in C^\infty(S)$.  With this in mind, let $\gamma \in C^\infty(I, S)$ be a smooth curve such that $\gamma(0)= x_0$ and $\gamma'(0)= v$. Then,
$$v(f) = \left.\frac{\rmd}{\rmd t}\right|_{t=0} f \circ \gamma(t)\,, \qquad V(f) = \left.\frac{\rmd}{\rmd t}\right|_{t=0} f \circ L_{x x_0} \circ \gamma(t)\,. $$
Furthermore, consider the map $\hat \gamma : S^{(2)}\times (- \epsilon, \epsilon) \rightarrow S^{(3)}$ defined as $(x,y, t) \mapsto (x,y, \gamma(t))$. This map is clearly smooth in $t$ as we have defined it in terms of a smooth curve.  Thus,
$$f \circ L_{x x_0} \circ \gamma(t) = (f \circ \mu \circ \hat \gamma)(x,y,t)\,,$$
is smooth as it is the composition of smooth maps, and so $V(f)$ is smooth. Left-invariance follows via direct calculation, i.e.,
$$V_{[x,y,z]} = (\rmd L_{[x,y,z]x_0})_{x_0} v = (\rmd L_{xy})_z \circ (\rmd L_{z x_0})_{x_0}v = (\rmd L_{xy})_z V_z\,.$$ 
Given that $L_{x_0 x_0} =  \Id_S$, it is evident that $V(x_0) = v$.
\end{proof}
In the other direction we have the following.
\begin{proposition}
Let $(S, \mu, x_0)$ be a biunital Lie semiheap. Then the value of any left-invariant vector field at an arbitrary point $x \in S$ is determined by its value at $x_0$.
\end{proposition}
\begin{proof}
From Proposition \ref{prop:BiLieReach}, we know that any point in $S$ can be reached from $x_0$ via a left-translation. Thus,
$$(\rmd L_{x x_0})_{x_0} :  \sT_{x_0}S \longrightarrow \sT_x S\,,$$
is well defined for all $x \in S$. The left-invariance of a vector field implies that 
$$X_x = (\rmd L_{x x_0})_{x_0} X_{x_0} \,,$$
and thus $X_{x_0}$ determines the value of $X_x$ for all points $x\in S$.
\end{proof}
Consider the maps
\begin{align*}
\Phi :\, & S \times \sT_{x_0}S \longrightarrow \sT S\\
& (x,v) \longmapsto (x, (\rmd L_{x x_0})_{x_0}v)\,.
\end{align*}
Clearly, the above is a smooth bundle map (over the identity on $S$), however this is not an isomorphism in general. Thus, the tangent bundle of a biunital Lie semiheap need not, in general, be trivialisable, i.e., biunital semiheaps are not necessarily parallelisible.\par 
Let us pick a basis $\{ e_\alpha \}_{\alpha = 1, \cdots, n}$ (not necessarily the coordinate basis) of $\sT_{x_0} S$ of a biunital Lie semiheap (assuming $\dim S = n$).  We can then construct the ``not everywhere vanishing'' left-invariant vector fields viz
$$E_\alpha(x) := (\rmd L_{x x_0})_{x_0}e_\alpha\,,$$
and build other left-invariant vector fields as linear combinations of these. However, we make no claim that this is a basis for the left-invariant vector fields.  For instance, these vector fields may be singular.   None-the-less, we have the following.
\begin{theorem}
Let $(S, \mu, x_0)$ be a bunital Lie semiheap. Then non-zero left-invariant vector fields exists. 
\end{theorem}
  Recall that we have an isomorphism of categories between the categories of Lie groups and pointed Lie heaps, see Theorem \ref{trm:GrpHeapIso}. We then have the following expected result. 
  \begin{proposition}
  Let $(G, m, i)$ be a Lie group. Then the set of left-invariant vector fields on $G$ coincides with the set of left-invariant vector fields on the associated pointed Lie heap $\mathcal{H}(G) =  (G,\mu, e)$. 
  \end{proposition}
  \begin{proof}
  Assume $X \in \Vect(G)$ is left-invariant on the associated Lie heap $\mathcal{H}(G)$, i.e., for all $x,y$ and $z\in G$, $(\rmd L_{xy})_z X_z =  X_{[x,y,x]}$. Given that $[x,y,z] = x\cdot  y^{-1} \cdot z$ implies that, setting $y =e$ (the group identity), $L_ {xe} =  L_x$ (on the right we mean left-translations in the group sense). Thus, 
  $$(\rmd L_x)_z X_z =(\rmd L_{x e})_z X_z  =  X_{[x,e,z]} = X_{x\cdot z}\,,$$
  and so $X$ is left-invariant in the group sense.  In other words, $\Vect_L(\mathcal{H}(G))\subset \Vect_L(G)$.\par 
  In the other direction, assume that $X$ is left-invariant in the group sense. Then sending $x \mapsto x\cdot y^{-1}$ for an arbitrary point $y \in G$, shows that 
  $$(\rmd L_{xy})_z X_z = (\rmd L_{x \cdot y^{-1}})_z X_z = X_{x\cdot y^{-1}\cdot z} = X_{[x,y,z]}\,,$$
  and so $\Vect_L(G)\subset \Vect_L(\mathcal{H}(G))$. In conclusion, as one is the subset of the other and vice versa, $\Vect_L(G)= \Vect_L(\mathcal{H}(G))\,.$
  \end{proof}
  As sets and indeed real vector spaces $\Vect_L(G)$ and  $\Vect_L(\mathcal{H}(G))$ are isomorphic. This is also true of the Lie algebras as the bracket is unchanged on the identification of the two vector spaces.  Given that the categories of  Lie groups and pointed Lie heaps are isomorphic  we make the following observation.
  \begin{corollary}
  The Lie algebra of left-invariant vector fields on a pointed Lie heap $(S, \mu, x_0)$ is of finite dimension $n = \dim S$, and is isomorphic to the Lie algebra $\mathfrak{g}$ of the associated Lie group $\mathcal{G}(S)$.
  \end{corollary}
  \subsection{Towards Multiplicative Structures} In this subsection we will make some preliminary definitions of what one could mean by multiplicative functions, forms and vector fields on a Lie semiheap. We will defer the question of their existence and further properties for now.  For a review of the notion of multiplicative structures on Lie groups and groupoids the reader can consult Kosmann-Schwarzbach \cite{Kosmann-Schwarzbach:2016}.
\par 
  Recall that $\R$ can be considered with its standard topology and smooth structure as a Lie (semi)heap with the ``$+ - +$'' operation, i.e.,
  $$[u,v,w]_\R := u-v+w\,.$$
  Functions on $S$ are, by definition, smooth maps $S \rightarrow \R$. We are thus led to the following.
  \begin{definition}
 Let $(S, \mu)$ be a Lie semiheap. A function $f \in \C^\infty(S)$ is said to me a \emph{multiplicative function} if it is a Lie semiheap homomorphism from $(S, \mu)$ to $(\R, [-,-,-]_\R)$, i.e.,
  $$f([x,y,z]) = f(x)- f(y) + f(z)\,,$$
  for all $x,y$ and $z  \in S$. If $S$ is a pointed Lie semiheap, then we further insist on the condition $f(x_0) =0$ for a function to be multiplicative. 
  \end{definition}
  \medskip 
  
  We notice that the zero function, i.e., $f(x) =0$ for all $x \in S$, is trivially multiplicative. As for $k$-forms, we make the following definition.
  \begin{definition}
  Let $(S, \mu)$ be a Lie semiheap. A $k$-form $\Theta \in \Omega^k(S)$ is said to be a \emph{multiplicative $k$-form} if 
  $$\mu^* \Theta = \textrm{pj}_1^*\Theta - \textrm{pj}_2^*\Theta + \textrm{pj}_3^*\Theta\,, $$
  where $\textrm{pj}_i : S^{(3)} \longrightarrow S$ is the projection onto the $i$-th factor.
  \end{definition}
 We directly observe that for multiplicative $0$-forms
 $$(\mu^*f)(x,y,z) =  f ([x,y,z]) = f(x)- f(y)+ f(z)\,,$$
 and so multiplicative $0$-forms are precisely multiplicative functions. \par
  Recall that given a vector field $X \in \Vect(X)$, its local flow is a smooth map 
  $$\Phi:  I \times S \longrightarrow S\,,$$
  where $ I \subset \R$ is an open interval containing $0$, such that
  \begin{enumerate}
  \item $\Phi(0, x) = x$ for all $x \in S$, and 
  \item $\left.\frac{\rmd}{ \rmd t}\right|_{t = t_0}\, \Phi(t,x) =  X_{\Phi(t_0,x)}$ for all $x  \in S$ and $t_0 \in I$.
  \end{enumerate}
  For convenience we write $\Phi_t(-):=  \Phi(t, -)$.
  \begin{definition}
  Let $(S, \mu)$ be a Lie semiheap. A vector field $S \in \Vect(S)$ is said to be a \emph{multiplicative vector field} if its local flow is a Lie semiheap homomorphism, i.e.,
  $$\Phi_t[x,y,z]  =  [\Phi_t(x), \Phi_t(y), \Phi_t(z)]\,,$$
  for all $t \in I$ and $x,y,z \in S $.
  \end{definition}
\subsection{The Ternary Coalgebraic Structure of Functions}
Recall that $C^\infty(S)$ is a nuclear Fr\'{e}chet algebra, and so $C^\infty(S) \, \widehat{\otimes}\,  C^\infty(S)\cong C^\infty(S\times S)$ with respect to any reasonable topology, for instance the  projective and injective topologies (see for example \cite[Part III]{Treves:2006}).    Similar statements hold for any finite number of (suitably completed) tensor products. When required, we will denote the multiplication map in $C^\infty(S)$ as $\mathrm{m} : C^\infty(S) \times C^\infty(S) \rightarrow C^\infty(S)$, and the unit map as $\eta : \star \rightarrow C^\infty(S)$ (the unit function $\Id_{C^{ \infty}(S)}$ is the constant function with value $1$, i.e., $\Id_{C^{ \infty}(S)}(x) =  1$ for all $x\in S$).
\begin{definition}
Let $(S, \mu)$ be a Lie semiheap. Then the associated \emph{canonical ternary comultiplication} $\Delta : \, C^{\infty}(S) \rightarrow C^\infty(S) \, \widehat{\otimes}\,  C^\infty(S)\, \widehat{\otimes} \, C^\infty(S) $
 is defined as $\Delta f := \mu^* f =  f \circ \mu$.
 \end{definition}
 More explicitly, for basic elements we can write $\Delta f(x_1\otimes x_2 \otimes x_3) = f \big ( [x_1, x_2, x_3] \big)$.
 \begin{example} 
 Consider a Lie group $G$ and its associated heap $S_G$. Then given any function $f$ on $S_G$ (and so on $G$) $\Delta f(g_1 \otimes g_2 \otimes  g_3) :=  f(g_1 g_2^{-1}g_3)$. This should be compared with the usual comultiplication on the algebra of functions on a Lie group.
 \end{example}
 \begin{proposition}
 Let $(S, \mu)$ be a Lie semiheap. Then the associated canonical ternary comultiplication satisfies the following properties.
 \begin{enumerate}
 \itemsep=10pt
 \item $\Delta$ is $\R$-linear;
 \item $\Delta(f_1 f_2) =  \Delta f_1 \, \Delta f_2$ for all $f_1, f_2 \in C^\infty(S)$;
 \item $\Delta \circ \eta = \eta^{(3)}$, where $\eta^{(3)}$ is the unit map for $C^\infty(S\times S \times S)$;
 \item The following identity holds
 \begin{equation}\label{eqn:ParaCoASS}
 (\Id^{(2)}\otimes \Delta)\circ \Delta = (\Id \otimes (\Delta \circ s_{12})\otimes \Id)\circ \Delta  = (\Delta \otimes \Id^{(2)})\circ \Delta\,.
 \end{equation}
 \end{enumerate}
 \end{proposition}
\begin{proof}\
\begin{enumerate}
\itemsep=10pt
\item $\R$-linearity is clear as $\Delta$ is defined by the pullback of the ternary multiplication.
\item Similarly, the pullback is an algebra homomorphism, so we have the result.
\item $\big (\Delta \circ \eta \big)(\textrm{pt})(x_1 \otimes x_2 \otimes x_3)= \Delta \Id_{C^\infty(S)}(x_1 \otimes x_2 \otimes x_3) =\Id_{C^\infty(S)}[x_1 , x_2 , x_3] =  1 $. As $x_i$ are arbitrary, we have the result.
 \item Let $f \in C^{\infty}(S)$ be an arbitrary smooth function. The para-associative law for the semiheap ternary multiplication  means 
$$f([x_1, x_2 , [x_3, x_4, x_5]]) = f([x_1, [x_4, x_3, x_2], x_5]) = f([[x_1, x_2, x_3], x_4, x_5])\, ,$$
which can be written as 
$$(\Id^{(2)}\otimes \Delta) \Delta f = (\Id \otimes (\Delta \circ s_{12})\otimes \Id)\Delta f  = (\Delta \otimes \Id^{(2)})\Delta f \,.$$
\end{enumerate}
\end{proof} 
 We refer to the identity \eqref{eqn:ParaCoASS} as \emph{para-coassociativity}. Note that, just as with associativity and coassociativity, para-coassociativity is para-associativity, but with the direction of arrows reversed.
 \begin{remark}
Ternary Hopf algebra were studied by Borowiec,  Dudek \&  Duplij   see \cite{Borowiec:2001}.  Note the structures they consider are not the same as presented here.  Quantum heaps, loosely, Hopf algebras without a counit, were introduced Škoda in \cite{Skoda:2007}.
 \end{remark} 
We then note that $\big(C^\infty(S), \textrm{m}, \eta, \Delta \big )$ consists of a unital associative (commutative) algebra  $\big(C^\infty(S), \textrm{m}, \eta \big )$ and a ``ternary para-coassociative coalgebra'' $\big(C^\infty(S),\Delta \big )$ such that $\Delta$ is a unital associative algebra homomorphism.  This should be compared with the definition of a bialgebra (see, for example \cite[Chapter 2]{Brzezinski:2003}). 
\begin{example}
The set $\R^n$ can be considered as either a vector space or a smooth manifold in the standard ways. Obviously, $\sT_p \R^n \cong \R^n$ for any point $p \in \R^n$. The standard scalar product on the vector space $\R^n$ induces a Riemannian metric on $\R^n$ (as a smooth manifold). That is, 
$$\langle u_p , v_p  \rangle =  u\cdot v = u^i v^j \delta_{ji} \,,$$
where $u = u^ie_i$ (etc.) in some chosen basis of $\sT_p \R^n$.  We view $(u^i , v^j)$ as global coordinates on $\R^n \times \R^n$. The smooth manifold $\R^n$ has the structure of a Lie semiheap, which can be conveniently specified using the pullback of the chosen coordinates.  Specifically 
\begin{align*}
\mu^* &:~ C^\infty(\R^n) \longrightarrow C^\infty(\R^n \times \R^n \times \R^n)\\
& \mu^* x^i =  u^i v^j w^k \delta_{kj}\,,
\end{align*}
were we have coordinates $x^i$ on $\R^n$ and $(u^j, v^k, w^l)$ on $\R^n \times \R^n \times \R^n \cong \R^{3n}$.
\end{example}
%

\subsection{Smooth Semiheap Actions}
The notion of an action of a Lie group on a manifold  generlises to Lie semiheaps. We make the following definition.
\begin{definition}
Let $M$ be a smooth manifold and $(S, \mu)$ be a Lie semiheap. A \emph{right action} of $S$ on $M$ is a smooth map
\begin{align*}
\sigma : &~ M \times S^{(2)} \longrightarrow M\\
&(m, x, y) \mapsto \sigma_{xy}(m)
\end{align*}
such that the compatibility condition
$$\sigma_{x_3 x_4}\circ \sigma_{x_1 x_2} = \sigma_{x_1[x_2, x_3, x_4] }\,,$$
holds. A smooth manifold $M$, equipped with a semiheap action will be referred to as a \emph{$S$-space}. 
\end{definition}
We will change notation slightly, where convenient, and set $m \triangleleft (x, y):= \sigma_{xy}(m)$, where $m \in M$ and $(x,y) \in S^{(2)}$.
\begin{remark}
The notion of a left semiheap action is clear. However, as we are interested in generalising principal bundles, right actions are more natural for our purposes. 
\end{remark}
A right semiheap action can be considered as a map $\tau : S^{(2)} \rightarrow \Hom_{\catname{Man}}(M,M)$. By fixing $x_1 =x$, we observe that the compatibility condition is described by a map $ \tau_x : S^{(3)} \rightarrow \Hom_{\catname{Man}}(M,M)$ given by $(x_1, x_2, x_3) \mapsto \tau(x, [x_1,x_2,x_3])$.

\begin{example}
The \emph{trivial action} of a Lie semiheap $S$ on a smooth manifold $M$ is defined by 
$$m \triangleleft (x, y) =m\,,$$
 for all $m \in M$ and $(x,y) \in S^{(2)}$.
\end{example}
\begin{example}
Let $S$ be a Lie semiheap. Then $S$ can be considered as a $S$-space via the right translation map (see Definition \ref{def:RLActions}).  Note that, in general, the right translation map $R_{x_1 x_2} :  S \rightarrow S$ is not a diffeomorphism. 
\end{example}
\begin{example}
Let $\psi : S \rightarrow S'$ be a homomorphism of Lie semiheaps.  We can then define an action of $S$ on $S'$ as 
\begin{align*}
& S' \times S^{(2)}  \longrightarrow S'\\
&(y, (x_1, x_2)) \mapsto [y, \psi(x_1), \psi(x_2)]'
\end{align*}
The observation that 
$$[[y, \psi(x_1), \psi(x_3)]', \psi(x_4), \psi(x_5)]' =[y, \psi(x_1),\psi[ x_3, x_4, x_5]]' \,,$$
which follows from para-associativity and the definition of a homomorphism of semiheaps, establishes that we have constructed an action in this way.
\end{example}
\begin{example} 
As a specific case of the above example, consider the affine line $\mathbb{A}$ equipped with it's heap structure $\{t_1, t_2, t_3\} = t_1-t_2 + t_3$. This is clearly a Lie heap with respect to the standard smooth structure. Let $(S, \mu)$ be an arbitrary Lie semiheap, and let $\varphi : \mathbb{A} \rightarrow S$ be a homomorphism of Lie semiheaps. Then we have a Lie heap action
\begin{align*}
& \sigma : \, S \times \mathbb{A}^{(2)} \longrightarrow S\\
& (x, (t_1, t_2)) \longmapsto [x,\varphi(t_1), \varphi(t_2)]
\end{align*}
\end{example}
\begin{definition}
Let $(S, \mu)$ be a Lie semiheap and let $M$ and $N$ be $S$-spaces. Then a smooth map $\psi :  M \rightarrow N$ is said to be \emph{$S$-equivariant} if for all $m \in M$ and $(x,y) \in S^{(2)}$
$$\psi \big( m \triangleleft (x,y) \big) = \psi(m)\triangleleft (x,y)\,.$$
\end{definition}
\begin{example}\label{exp:GSpace}
Let $G$ be a Lie group and $M$ be a right $G$-space. We denote the action $M \times G  \rightarrow M$ as $(m,g) \mapsto a_g(m)$. One can build a (semi)heap action $M \times G^{(2)} \rightarrow M $ by setting
$$(m,g_1 , g_2) \longmapsto a_{g_1^{-1}g_2} =: \sigma_{g_1 g_2}\,,$$
and the ternary product is defined as $[g_1, g_2, g_3]:= g_1 g_2^{-1}g_3$. To show that we do indeed have an action, observe that
$$\sigma_{g_3 g_4} \circ \sigma_{g_1 g_2} = a_{g_3^{-1}g_4}\circ a_{g_1^{-1}g_2}= a_{g_1^{-1}g_2 g_3^{-1}g_4} = a_{g_1^{-1}[g_2, g_3, g_4]} = \sigma_{g_1[g_2, g_3, g_4]}\,. $$
Let $N$ be a another $G$-space equipped with the heap action as above. If $\psi : M \rightarrow N$ is a $G$-equivariant map, then it is also $S_G$-equivariant.
\end{example}
\begin{example}
  As a specific example of the previous example,  recall that a flow on a smooth manifold $M$ is a smooth action of the additive group of real number $(\R, +)$ 
$$\varphi : M \times \R \longrightarrow M\,,$$
such that for all $m \in M$ and, $t_1$ and $t_2 \in \R$ 
\begin{align*}
& \varphi(m,0)= m\,,
& \varphi(\varphi(m, t_1), t_2) =  \varphi(m, t_1 + t_2)\,.
\end{align*} 
The group structure can be replaced by the heap structure $[t_1, t_2, t_3] = t_1 - t_2 + t_3$. A Lie heap action 
\begin{align*}
& \sigma : M \times \R^{(2)} \longrightarrow M\\
& (m, (t_1, t_2)) \longmapsto \varphi(m, - t_1 + t_2)\,.
\end{align*}
\end{example}

The category of $S$-spaces is evident and we denote it by $\catname{Man}_S$, with $\textnormal{Ob}(\catname{Man}_S)$ being (smooth) $S$-spaces and $\textnormal{Hom}(\catname{Man}_S)$ being $S$-equivariant maps.\par 
The \emph{orbit of an element} $m \in M$ is the set of points that can be reached from $m$ using the elements of $S$, i.e., 
$$m \triangleleft S^{(2)} :=  \{ m \triangleleft (x_1, x_2) ~~ |~~ (x_1, x_2)\in S^{(2)} \}\,.$$
However, like semigroup and monoid actions, we do not, in general, have an associated equivalence relation. Thus, we cannot construct the orbit set using equivalence classes as one would with group actions. This needs to be taken into account with the starting definition with semiheap bundles.

\section{Semiheap and Principal Bundles}
\subsection{Semiheap Bundles}
We now proceed to mimic as closely as possible the definition of a principal bundle in terms of a fibre bundle with a compatible group action (see for example \cite[Chapter 9.4]{Nakahara:2003}), but now in the setting of Lie semiheaps. Our approach is to consider a $S$-space together with a compatible local trivialisation. 
\begin{definition}\label{def:SemiHeapBund}
A \emph{semiheap bundle} consists of the following.
\begin{enumerate}
 \itemsep=10pt
\item A $S$-space $P$;
\item A surjective submersion $\pi : P \rightarrow M$, such that the action of $S$ on $M$ is trivial;
\item An open cover $\{ U_i\}_{i \in \mathcal{I}}$ of $M$ and a collection of $S$-equivariant diffeomorphisms 
$$t_i : \pi^{-1}(U_i) \stackrel{\sim}{\rightarrow}U_i \times S\,,$$
 where the action on $U_i \times S$ is right translation on the Lie semiheap, such that the following diagram is  commutative
 \begin{center}
\leavevmode
\begin{xy}
(0,15)*+{\pi^{-1}(U_i)}="a"; (30,15)*+{U_i \times S}="b";%
(30,0)*+{U_i}="c";%
{\ar "a";"b"}?*!/_3mm/{t_i };%
{\ar "a";"c"}?*!/^3mm/{\pi};{\ar "b";"c"}?*!/_5mm/{\textrm{prj}_1};%
\end{xy}
\end{center}
\end{enumerate}
We will denote a semiheap bundle as a triple $(P,M,S)$. The collection $\{(U_i, t_i) \}_{i \in \mathcal{I}}$ we refer to as a \emph{local equivariant trivialisation}.
\end{definition}
Note that the action is trivial on $M$, and so preserves the fibres, i.e., if $\pi(p) = m$, then $\pi(p \triangleleft (x,y))=m$. 
\begin{remark}
We have no notion of a free action as there is no identity element. An action is transitive if for every pair of points $p,q \in \pi^{-1}(m)$, there exists a pair $x,y \in S$ such that $p \triangleleft (x,y) =q$.  We will not insist on transitivity in our definition of a semiheap bundle. This should be compared with the definition of a principal bundle.
\end{remark}
\begin{example}
Any Lie semiheap can be considered as a semiheap bundle over a single point, $\{ m\} \times S \rightarrow \{ m\}$, where  the action is the right translation. 
\end{example}
\begin{example}
A \emph{trivial semiheap bundle} is the Cartesian product $P = M \times S$, where $M$ is a smooth manifold and $S$ a Lie semiheap, together with the canonical projection onto the first factor. The action of $S$ on $P$ is simply the right action, i.e., $(m,x,x_1, x_2)\mapsto (m, [x, x_1, x_2])$.  
\end{example}
\begin{definition}
Let $(P, M,S)$ and $(P', M', S')$ be semiheap bundles. Then a \emph{semiheap bundle homomorphism} is a pair $(\Phi, \psi)$, where $\psi : S \rightarrow S'$ is a Lie semiheap homomorphism, and $\Phi: P \rightarrow P'$ is a (smooth) bundle map (over $\phi : M \rightarrow M'$) that is $\psi$-equivariant in the sense that
$$\Phi(p \triangleleft (x,y)) = \Phi(p) \triangleleft (\psi(x), \psi(y))\,,$$
where $p \in P$, and $x,y \in S$.
\end{definition}
In this way, we obtain the category of semiheap bundles, which we denote as $\catname{SemiBun}$. The objects, $\textnormal{Ob}(\catname{SemiBun})$ are semiheap bundles, and the homomorphisms, $\textnormal{Hom}(\catname{SemiBun})$ are $\psi$-equivariant maps. If $S = S'$ and $\psi = \Id_S$, then we obtain the subcategory of $S$-bundles, which we denote as $\catname{SemiBun}_S$.
\begin{proposition}
Let $(P, M,S)$ be a semiheap bundle. Then each fibre $F_m := \pi^{-1}(m)$ is non-canonically isomorphic as a Lie semiheap to $S$.  
\end{proposition}
\begin{proof}
Let $\{(U_i, t_i) \}_{i \in \mathcal{I}}$ be a local equivariant trivialisation of $(P, M,S)$ and consider $p \in \pi^{-1}(U_i)$ (we set $\pi(p) =m$). Clearly, $\pi^{-1}(m) =: F_m \stackrel{~}{\rightarrow} \{ m\} \times S \cong S$, specifically $t_i(p) = (m,x) \cong x $ as the point $m \in U_i$ is fixed. As standard, the fibre at any point is non-canonically diffeomorphic to $S$.\par 
 From Proposition \ref{prop:InducedSemiHeap} we know how to proceed.  Let $p,q$ and $r \in F_m$ be arbitrary points. Assume that $m \in U_i$. We then define a ternary operation using the local trivialisation as $[p,q,r]_i := t_i^{-1}[t_i(p), t_i(q), t_i(r)]$.  We know this is  an induced Lie semiheap structure on the fibre at $m$. Picking another local trivialisation, say $(U_j, t_j)$, with $m \in U_j$, we know that $t_j^{-1} t_i$ is a canonical isomorphism between the two induced Lie semiheap structures.
\end{proof}
The above proposition tells us that a semiheap bundle is a smooth family of Lie semiheaps for which each member is (non-canonically) diffeomorphic to a given Lie semiheap.
\begin{example}
Consider a Euclidean vector bundle $(E, g)$ of rank $q$. By employing an orthonormal trivialisation $\{ (U_i, t_i)\}_{i \in \mathcal{I}}$ each
$$\pi^{-1}(U_i) \stackrel{t_i}{\longrightarrow} U_i \times \R^q$$
are isometries where $\R^q$ is equipped the standard Euclidean structure, which we denote as $\delta$. Then $\R^q$  can be considered as a semiheap by defining $[x_1, x_2, x_3] := x_1 \, \delta(x_2, x_3)$. Similarly, each fibre can be considered as a semiheap using $g_m$. An action on $E$ can be defined fibrewise  as
\begin{align*}
& E_m \times \R^q \times \R^q  \longrightarrow E_m\\
& (v, x_1, x_2)  \mapsto v \delta(x_1, x_2) = v g_m(t_i^{-1}(x_1),t_i^{-1}(x_2) )\,.
\end{align*}
 To check this is a semiheap action we observe that
 $$g_m(t_i^{-1}(x_1),t_i^{-1}(x_2))g_m(t_i^{-1}(x_3),t_i^{-1}(x_4)) = g_m(t_i^{-1}(x_1),t_i^{-1}(x_2) \, g_m(t_i^{-1}(x_3),t_i^{-1}(x_4)))\,.$$
Note this action is smooth, as it is built from smooth operations, and is trivial on the base $M$.  Furthermore, note that $\delta(x_1, x_2) =  g_m(t_i^{-1}(x_1),t_i^{-1}(x_2) ) = g_m(t_j^{-1}(x_1),t_j^{-1}(x_2) )$, and so the action is well-defined.  Thus, any Euclidean vector bundle can then be considered as a semiheap bundle with the semiheap being the Euclidean space considered as a semiheap.
\end{example}

\subsection{Principal Bundles as Semiheap Bundles}
We now proceed to generalise the heapification functor to the setting of bundles. In particular, principal bundles will provide a class of semiheap bundles, thus showing the category contains interesting and useful objects. 
\begin{proposition}
Any principal bundle $(P, M,G)$ is canonically associated with the semiheap bundle $(P, M, S_G)$.
\end{proposition}
\begin{proof}
Let $(P,M,G)$ be a principal bundle, and we denote the principal action (which is free and transitive) as $(p,g)\mapsto a_g(p)$. We can build a semiheap action via Example \ref{exp:GSpace}, that is, we set $p \triangleleft (g_1, g_2) := a_{g_1^{-1}g_2}(p)$ and $[g_1,g_2, g_3]:=  g_1 g_2^{-1}g_3$. Thus, $P$ is a $S_G$-space,  and so we have the first part of Definition \ref{def:SemiHeapBund}. The second part is automatic as $\pi : P \rightarrow M$ is a surjective submersion. By definition, we have a local $G$-equivariant trivialisation of a principal bundle, $\{(U_i, t_i )\}_{i \in \mathcal{I}}$. Let us consider a given point $p \in P$ and set $t_i(p) = (m,g)$.  Then, using the $G$-equivalence of each $t_i$
$$t_i(p \triangleleft  (g_1, g_2)) =  t_i\big (a_{g_1^{-1}g_2}(p) \big) =   \big(m, gg_1^{-1}g_2 \big)  = \big(m, g \triangleleft (g_1,g_2) \big)\,. $$ 
Thus, each $t_i$ is $S_G$-equivariant, and so we have the third part of Definition \ref{def:SemiHeapBund}.
\end{proof}
\begin{proposition}
Canonically associated with any homomorphism of principal bundles $(\Phi, \psi) :  (P, M, G) \rightarrow (P', M', G')$ is a homomorphism of semiheap bundles $(P, M, S_G) \rightarrow (P', M', S_{G'})$.
\end{proposition}
\begin{proof}
From Definition \ref{def:HeapFunc}, we see that the Lie group homorphism $\psi : G \rightarrow G'$ is also a Lie semiheap homomorphism.  The equivariance of the map $\Phi$ shows that 
$$\Phi\big(p \triangleleft (g_1, g_2)\big) =  \Phi\big(a_{g^{-1}_1 g_2}(p)\big) = a'_{\psi(g_1)^{-1} \psi(g_2)}(p) = \Phi(p)\triangleleft (\psi(g_1), \psi(g_2))\,,$$
and thus we canonically have a homomorphism of semiheap bundles. 
\end{proof}
The previous two propositions led us to the following definition.
\begin{definition}
The \emph{bundle heapifiction functor} is the functor
$$\mathcal{H}: ~\catname{Prin}  \rightarrow \catname{SemiBun}\,$$
that acts on objects as $(P, M, G ) \mapsto (P, M ,S_G)$ and on homomorphisms it acts as the identity. 
\end{definition}
It is clear that the bundle heapification functor is faithful, however, it is not full. There are more morphisms as semiheap bundles than as principal bundles. We note that in the neighbourhood of every point of $M$, the fibres of a principal bundle can be given the structure of the group $G$ by choosing an element in each fibre to be the identity element. Thus, the fibres are not canonically pointed.  There is, in general, no privileged canonical point associated with an arbitrary principal bundle. This means that there is no immediately obvious way to `force' the bundle heapification functor to be full as we have done for the heapification functor for Lie groups.

\section{Concluding Remarks}
We have made an initial study of Lie semiheaps and semiheap bundles. Importantly, we have constructed heapification functors that shows that Lie groups and principal bundles provide natural examples of Lie semiheaps and semiheap bundles, respectively.\par 
In this introductory paper, we have only made an initial study of ternary operations, and in particular semiheaps, in differential geometry. There are plenty of open questions here:
\begin{itemize}
\item Can we find further examples of Lie semiheaps and bundles that are not directly connected to Lie groups and principle bundles?
\item Can one say more about the existence of left-invariant vector fields on general Lie semiheaps?
\item How much of the theory of connections on principal bundles generalises to the setting semiheap bundles?
\item Can Lie semiheaps be used to describe generalised symmetries in geometric mechanics, for example?
\end{itemize}
More generally, the r\^{o}le ternary operations in differential geometry has hardly been explored. We hope, in part, to rectify this in future publications. 
\section*{Acknowledgements} 
The author thanks Tomasz Brzeziński for introducing him to ternary operators and comments on earlier drafts of this work.


\begin{thebibliography}{10}
\begin{small}

\bibitem{Baer:1929}
Baer, R., Zur Einführung des Scharbegriffs, \href{https://doi.org/10.1515/crll.1929.160.199}{\emph{J. Reine Angew. Math.}} \textbf{160} (1929), 199--207.

\bibitem{Borowiec:2001}
Borowiec, A., Dudek W.A. \&  Duplij S.
Basic concepts of ternary Hopf algebras,
\emph{Journal of Kharkov National University, ser. Nuclei, Particles and Fields}, v. 529, N \textbf{3}(15) (2001) pp. 21--29.

\bibitem{Breaz:2022}
Breaz,S.,  Brzeziński,T.,  Rybołowicz, B. \&  Saracco, P,
Heaps of modules and affine spaces, \href{https://arxiv.org/abs/2203.07268}{\texttt{arXiv:2203.07268 [math.RA]}}.


\bibitem{Bruce:2022}
Bruce, A.J., 
Semiheaps and Ternary Algebras in Quantum Mechanics Revisited,
\href{https://doi.org/10.3390/universe8010056}{\emph{Universe}}  \textbf{8}(1) (2022), 56.

\bibitem{Brzezinski:2018}
Brzeziński, T., 
Towards semi-trusses,
{\emph{Rev. Roumaine Math. Pures Appl.}} \textbf{63} (2018), no. 2, 75--89. 

\bibitem{Brzezinski:2019}
Brzeziński, T.,
Trusses: between braces and rings,
\href{https://doi.org/10.1090/tran/7705}{\emph{Trans. Amer. Math. Soc.}} \textbf{372} (2019), no. 6, 4149--4176. 

\bibitem{Brzezinski:2020}
Brzeziński, T.,
Trusses: paragons, ideals and modules,
\href{https://doi.org/10.1016/j.jpaa.2019.106258}{\emph{J. Pure Appl. Algebra}} \textbf{224} (2020), no. 6, 106258, 39 pp. 

\bibitem{Brzezinski:2022}
Brzeziński, T., 
 Lie trusses and heaps of Lie affebras,
 \href{https://doi.org/10.48550/arXiv.2203.12975}{\texttt{arXiv:2203.12975 [math.RA]}}.
 
\bibitem{Brzezinski:2022b}
Brzeziński, T.,
The Algebra of Elliptic Curves,
\href{ https://doi.org/10.48550/arXiv.2209.05065}{\texttt{arXiv:2209.05065 [math.RA]}}.


\bibitem{Brzezinski:2019b}
Brzeziński, T. \& Rybołowicz, B.,
 Modules over trusses vs modules over rings: Direct sums and free modules,  \href{https://doi.org/10.1007/s10468-020-10008-8}{\emph{Algebr. Represent. Theor.}} (2020).
 
 
\bibitem{Brzezinski:2003}
Brzeziński, T. \& Wisbauer, R.,
\href{https://www.cambridge.org/gb/academic/subjects/mathematics/algebra/corings-and-comodules?format=PB&isbn=9780521539319}{Corings and comodules},
London Mathematical Society Lecture Note Series 309. Cambridge: Cambridge University Press (ISBN 0-521-53931-5/pbk). xi, 476 p. (2003). 

 


\bibitem{Grabowska:2003}
Grabowska, K., Grabowski, J. \&  Urbański, P.,
 Lie brackets on affine bundles,
  \href{ https://doi.org/10.1023/A:1024457728027}{\emph{Ann. Global Anal.Geom.}} \textbf{24} (2003), 101--130.
  
  \bibitem{Grabowski:2016}
  Grabowski, J.,
  An introduction to loopoids, 
 \href{http://dx.doi.org/10.14712/1213-7243.2015.184}{\emph{Comment. Math. Univ. Carolin.}} \textbf{57} (2016), 515--526.
  
\bibitem{Grabowski:2021}
Grabowski, J. \& Ravanpak Z., 
Nonassociative analogs of Lie groupoids,
\href{https://doi.org/10.48550/arXiv.2101.07258}{\texttt{arXiv:2101.07258[math.DG]}}.
 

\bibitem{Hollings:2017}
Hollings, C.D. \& Lawson, M.V., \href{https://doi.org/10.1007/978-3-319-63621-4}{Wagner’s theory of generalised heaps}, Springer, Cham, 2017. xv+189 pp. ISBN: 978-3-319-63620-7.

\bibitem{Kerner:2008}
Kerner, R.,
Ternary and non-associative structures, 
\href{https://doi.org/10.1142/S0219887808003326}{\emph{Int. J. Geom. Methods Mod. Phys.}} \textbf{5} (2008), 1265--1294.


\bibitem{Kerner:2018}
Kerner, R.,
Ternary generalizations of graded algebras with some physical applications, \emph{Rev. Roum. Math. Pures Appl.} \textbf{63} (2018), 107--141. 

\bibitem{Kolar:1993}
Kolář, I., Michor, P.W. \& Slovák, J.,
\href{https://doi.org/10.1007/978-3-662-02950-3}{Natural operations in differential geometry}, \emph{Springer-Verlag, Berlin}, 1993, vi+434 pp. ISBN: 3-540-56235-4.

\bibitem{Konstantinova:1978}
Konstantinova, L. I., 
 Semiheap bundles (Russian), \emph{Theory of semigroups
and its applications} No. 4 pp 46--54 Saratov Gos. Univ. Saratov 1978

\bibitem{Kontsevich:1999}
Kontsevich, M., 
Operads and Motives in Deformation Quantization,
 \href{https://doi.org/10.1023/A:1007555725247}{\emph{Letters in Mathematical Physics}} \textbf{48}, 35--72 (1999). 
 
 \bibitem{Kosmann-Schwarzbach:2016}
Kosmann-Schwarzbach, Y.,
Multiplicativity, from Lie groups to generalized geometry, \href{https://doi.org/10.4064/bc110-0-10 }{\emph{Banach Center Publ.}} \textbf{110} (2016), 131--166.
 
\bibitem{MacLane:1988}
Mac Lane, S.,
\href{https://doi.org/10.1007/978-1-4612-9839-7}{Categories for the working mathematician}, 2nd ed.
Graduate Texts in Mathematics, 5, New York, NY: Springer. xii, 314 p, (1998). 

\bibitem{Prufer:1924}
Pr\"{u}fer, H.,
Theorie der Abelschen Gruppen,
\href{https://doi.org/10.1007/BF01188079}{\emph{Math. Z.}} \textbf{20} (1924), no. 1, 165--187.

\bibitem{Nakahara:2003}
Nakahara, M.,
\href{https://doi.org/10.1201/9781315275826 }{Geometry, topology and physics}, 2nd ed,
Graduate Student Series in Physics. Bristol: Institute of Physics (IOP) xxii, 573 p. (2003). 

\bibitem{Saito:2021}
Saito, M. \& Zappala E.,
Braided Frobenius algebras from certain Hopf algebras,
\href{https://doi.org/10.1142/S0219498823500123 }{\emph{Journal of Algebra and Its Applications}} (Online Ready).

\bibitem{Skoda:2007}
Škoda, Z.,
Quantum heaps, cops and heapy categories,
\href{https://hrcak.srce.hr/12603}{\emph{Mathematical Communications}}, Vol. 12 No. 1, 2007.


\bibitem{Treves:2006}
Trèves, F., \href{https://www.doverbooks.co.uk/topological-vector-spaces-distributions-and-kernels}{Topological Vector Spaces, Distributions and Kernels}, Mineola, N.Y., (2006) [1967], Dover Publications.



\end{small}
\end{thebibliography}
\end{document}